\documentclass[11pt, reqno]{amsart}

\usepackage{graphicx}
\usepackage{amsmath,amssymb,mathtools,amsopn,amsfonts,amsthm}
\usepackage{bbm}
\usepackage{subfigure}
\usepackage{eepic}
\usepackage{tikz}
\usetikzlibrary{calc,angles,positioning,intersections,quotes}
\usepackage{a4wide}
\usepackage[utf8]{inputenc}
\usepackage{epsfig}
\usepackage{latexsym}
\usepackage[shortlabels]{enumitem}
\usepackage[UKenglish]{babel}
\usepackage{verbatim}
\usepackage[initials]{amsrefs}
\usepackage[bookmarks]{hyperref}
\usepackage{float}

\newtheorem{thm}{Theorem}[section]
\newtheorem{lma}[thm]{Lemma}

\theoremstyle{definition}
\newtheorem{defn}[thm]{Definition}

\newtheorem{conj}[thm]{Conjecture}

\newtheorem{ques}[thm]{Question}

\newtheorem{eg}[thm]{Example}

\newcommand{\R}{\mathbb{R}}
\renewcommand{\P}{\mathbb{P}}

\newcommand{\N}{\mathbb{N}}

\newcommand{\SL}{\mathrm{SL}(2,\mathbb{R})}

\providecommand{\norm}[1]{\lVert#1\rVert}

\newcommand{\I}{\mathcal{I}}

\newcommand{\A}{\mathcal{A}}

\renewcommand{\i}{\mathtt{i}}
\renewcommand{\j}{\mathtt{j}}

\newcommand{\hd}{\dim_\textup{H}}

\newcommand{\ld}{\dim_\textup{LY}}
\newcommand{\locd}{\dim_\textup{loc}}

\title{Self-projective sets}

\author{Argyrios Christodoulou} \address{Department of Mathematics, Aristotle University of Thessaloniki, Greece, 54124} 
\email{argyriac@math.auth.gr}

\author{Natalia Jurga} \address{Mathematical Institute, University of St Andrews, Scotland, KY16 9SS}
\email{naj1@st-andrews.ac.uk}

\begin{document}

 \maketitle

\begin{abstract}
Self-projective sets are natural fractal sets which describe the action of a semigroup of matrices on projective space. In recent years there has been growing interest in studying the dimension theory of self-projective sets, as well as progress in the understanding of closely related objects such as Furstenberg measures. The aim of this survey is twofold: first to motivate the study of these objects from several different perspectives and second to make the study of these objects more accessible for readers with expertise in iterated function systems.
\end{abstract}

\section{Introduction}

Self-projective sets are natural fractal sets which describe the action of a semigroup of matrices $G \subset \mathrm{SL}(d,\R)$ on $d-1$ dimensional projective space $\R\P^{d-1}$, which throughout the article is equipped with the metric induced by angles. In recent years there has been growing interest in studying the dimension theory of self-projective sets, as well as progress in the understanding of closely related objects such as Furstenberg measures. 

Self-projective sets are interesting from a variety of different perspectives and as such, their study has been scattered across literature coming from several different fields, such as complex analysis and hyperbolic geometry, fractal geometry, random matrix products and linear cocycles. The first aim of this survey is to collect these different perspectives together and to give an insight into the motivation for the study of these sets in different fields as well as concrete examples of interest.

From the point of view of fractal geometry, self-projective sets resemble attractors of iterated function systems. While self-projective sets cannot be studied through this lens directly (since they are limit sets of homeomorphisms of a compact space, hence not contractions) nevertheless they share many features with attractors of iterated function systems and indeed similar tools may be used to study them. Thus the second aim of this survey is to make self-projective sets more accessible to those with expertise in iterated function systems and inspire further work in this area by gathering some open problems. We will assume that the reader is familiar with basic notions of dimension theory, such as the Hausdorff dimension. Moreover, some expertise in the literature on self-similar sets and measures will be useful when reading this survey, especially when reading \S \ref{section:alg}.

Let $G \subset \mathrm{SL}(d,\R)$ be the semigroup generated by a set $\A \subset  \mathrm{SL}(d,\R)$. The following definition makes sense whether $\A$ is finite or infinite, however we will mostly be considering the finite case throughout this survey. In order to define self-projective sets we need to consider $\R G = \{kA\colon k\in\R\:\text{and}\:A\in G\}$ and its closure $\overline{\R G}$ in the set of $d\times d$ matrices, with respect to the usual topology. 

\begin{defn}
The limit set $\Lambda_\A$ of $\A$ (equiv. $G$) is the projection in $\R\P^{d-1}$ of the set
\[
\{\mathrm{Im}\pi \colon \pi\in \overline{\R G} \:\text{and}\:\mathrm{rank}(\pi)=1\}.
\]
\end{defn}

In principle, $\Lambda_\A$ can be the empty set. However, if we assume that $G$ contains a proximal element, that is, an element $A \in G$ with a simple leading eigenvalue $\lambda$, then $\lim_{n\to\infty}\lambda^{-n}A^n$ converges to an endomorphism of $\R^d$ with rank one. Hence, $\Lambda_\A$ is non-empty.\\
If, in addition, we assume that $G$ is irreducible, that is, there is no proper linear subspace $L \subset \R^d$ which is preserved by all elements in $G$, then $\Lambda_\A$ is the unique closed, $G$-invariant, subset of $\R\P^{d-1}$ whose $G$-orbits are dense in itself\footnote{Closed sets with dense orbits are often called minimal.} (see \cite[Lemma 4.2]{BQ}).

So, for the rest of the article, we shall assume that $G$ is an irreducible semigroup that contains proximal elements. In \S\ref{higher} we are going to show that under these assumptions the limit set has the following simple form, where the attracting fixed point of a proximal matrix is its fixed point in $\R\P^{d-1}$ that corresponds to its leading eigenvalue.

\begin{thm}\label{ls=attr}
If $G$ is an irreducible semigroup that contains a proximal matrix, then its limit set is the closure (in $\R\P^{d-1}$) of set of all attracting fixed points of proximal elements of $G$.
\end{thm}

A notion of limit set can still be defined in a similar way when $\overline{\R G}$ does not contain any endomorphisms of rank one. In this case, however, it may not have any of the useful properties mentioned above; for example, it may not be the unique minimal set. For precise statements see the discussion in \cite[Section 4.1]{BQ}. It is also worth mentioning that this more general notion could be used to define the limit set of the semigroup $G^{-1}=\{A^{-1}\colon A\in G\}$, which acts as a \emph{repeller} for $G$ (i.e. an attractor of its backward orbits). The interplay between the limit set and the repeller can yield insights into the dynamical behaviour of $G$ (see, for example, \cite[Section 15]{JS}), but is beyond the scope of this survey.

\subsection{Limit sets of Fuchsian semigroups}
Possibly the simplest occurrence of limit sets arises from the action of $\mathrm{SL}(2,\R)$ on the real projective line $\R\P^1$. This case is special in the sense that the study of self-projective sets heavily overlaps with another well-developed field of research; Fuchsian groups.

To see the connection between the two one only needs realise $\R\P^1$ as the extended real line $\overline{\R}=\R\cup\{\infty\}$, where infinity here is simply the north pole of the Riemann sphere. In this setting, the projective map induced by the action of a matrix 
\[
A=\begin{pmatrix}
a & b\\
c & d
\end{pmatrix}\in\mathrm{SL}(2,\R)
\]
on $\R\P^1$ is the M\"obius transformation $f_A(z)=\frac{az+b}{cz+d}$. So, a set of matrices $\A\subset\mathrm{SL}(2,\R)$ can be thought of as set of M\"obius transformations from $\mathrm{PSL}(2,\R)$.

The study of the dynamical properties of semigroups of $\mathrm{PSL}(2,\R)$ was initiated by Hinkkanen and Martin \cite{hinkkanen-martin} as a special case in the theory of rational semigroups, and further developed by Fried, Marotta and Stankiewitz \cite{FrMaSt}, and Jacques and Short \cite{JS}. The main motivation behind the results in these last two articles is the rich theory of Fuchsian groups; discrete subgroups of $\mathrm{PSL}(2,\R)$ \cite{beardon,katok, marden}.

Limit sets of semigroups of $\mathrm{PSL}(2,\R)$ are typically defined through the action of M\"obius maps on two-dimensional hyperbolic space. But the work of the authors mentioned above showed that, under our standing assumptions (translated to the setting of M\"obius maps), this definition coincides with the one presented in this article (see \cite[Section 6]{JS} for more information). This connection paved the way for the introduction of techniques from hyperbolic geometry into the theory of self-projective sets \cite{cj} and two-dimensional linear cocycles \cite{c}.

\subsection{Rauzy gasket}

The Rauzy gasket is a well-known fractal which can be obtained as the limit set of a projective system of three matrices in $\mathrm{SL}(3,\R)$. It was discovered independently three times and since then there has been substantial work towards establishing that it has zero Lebesgue measure and toward estimating its Hausdorff dimension. Its numerous applications motivate a more comprehensive investigation into the study of projective systems associated to subsets of $\mathrm{SL}(3,\R)$.

\begin{figure}[h!]
\centering
\includegraphics[width=5cm]{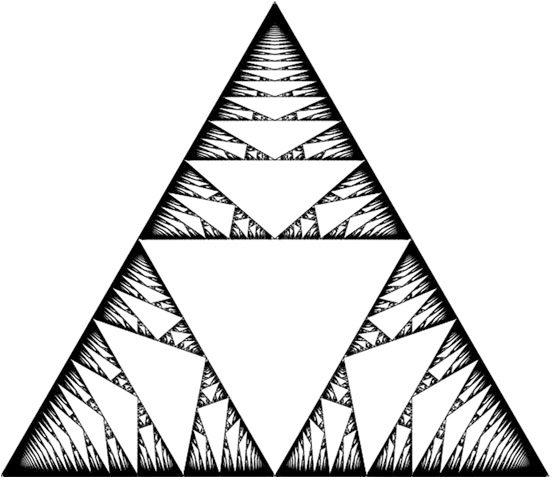}
\caption{The Rauzy gasket contained in the standard 2-simplex $\Delta$}
\end{figure}

The matrices which generate the projective system for the Rauzy gasket are
\begin{align}
\A=\left\{ \begin{pmatrix} 1& 1&1 \\ 0 & 1&0\\ 0&0&1\end{pmatrix}, \begin{pmatrix} 1& 0&0 \\ 1 & 1&1\\ 0&0&1\end{pmatrix}, \begin{pmatrix} 1& 0&0 \\ 0 & 1&0\\ 1&1&1\end{pmatrix}\right\}. \label{rauzy}
\end{align}
Since the matrices are non-negative, $\Lambda_\A$ is contained in the closed cone $K \subset \R\P^2$ of non-negative directions and therefore the projective system induced by $\A$ is actually an IFS on $K$. 

Typically, the Rauzy gasket is defined to be the attractor $R$ of the IFS $\{f_1,f_2,f_3: \Delta \to \Delta\}$ induced by $\A$ on the standard 2-simplex $\Delta=\{(x,y,z): x,y,z \geq 0, \; x+y+z=1\}$ where
\begin{align*}
f_1(x,y,z)&=\left(\frac{1}{2-x}, \frac{y}{2-x}, \frac{z}{2-x}\right)\\
f_2(x,y,z)&=\left(\frac{x}{2-y}, \frac{1}{2-y}, \frac{z}{2-y}\right)\\
f_3(x,y,z)&=\left(\frac{x}{2-z}, \frac{y}{2-z}, \frac{1}{2-z}\right).
\end{align*}
It is easy to see that $R$ is a bi-Lipschitz image of $\Lambda_\A$.

The Rauzy gasket has a long and interesting history. It first appeared in work of Arnoux and Rauzy in the context of interval exchange transformations in \cite{ar}, where it was conjectured to have zero Lebesgue measure. The gasket was rediscovered by Levitt \cite{levitt} where it was associated with the simplest example of pseudogroups of rotations, and the paper included a proof (accredited to Yoccoz) that the set has zero Lebesgue measure. Later, the gasket was reintroduced in the work of Dynnikov and De Leo \cite{dd}  in connection with Novikov’s problem of plane sections of triply periodic surfaces. Here an alternative proof was given that the gasket has zero Lebesgue measure, and it was conjectured that the Hausdorff dimension should be strictly less than 2. This was rigorously established by Avila, Hubert and Skripchenko \cite{ahs}. Since then there have been a number of further papers estimating the Hausdorff dimension of the gasket \cite{fougeron, sewell,grm} and very recently the Hausdorff dimension of the Rauzy gasket was established independently by the second named author \cite{rauzyme} and Jiao, Li, Pan and Xu \cite{jiao2023dimension, li2023dimension}, see \S \ref{dimhigher}. In an earlier edition of these proceedings \cite{as} the Rauzy gasket was studied again, this time as a subset of the standard 2-dimensional simplex associated with letter frequencies of ternary episturmian words. Recently Gamburd, Magee and Ronan \cite{gmr} showed asymptotic estimates for the number of integer solutions of the Markov-Hurwitz equations, which featured the dimension of the Rauzy gasket.

\subsection{Connection to the Furstenberg measure}

Here we discuss the connection between the limit set and a closely related object, the Furstenberg measure.

Let $\A=\{A_1, \ldots, A_N\} \subset\mathrm{SL}(d,\R)$ generate a proximal and strongly irreducible semigroup $G$, where by strongly irreducible we mean that there is no finite union of proper linear subspaces of $\R^d$ which is preserved by all elements in $G$, and by proximal we mean that $G$ contains a proximal element. Let $\mu$ be a measure on $\A$ which gives positive mass to each atom. So $\mu= \sum_{i=1}^N p_i \delta_{A_i}$ where $p_i>0$ for each $i$ and $p_1 + \cdots +p_N=1$. By the work of Furstenberg there exists a unique $\mu$-stationary measure $\nu$ on $\R\P^{d-1}$:
\begin{equation}
\nu=\sum_{i=1}^N p_i A_i \nu,
\label{stationary}
\end{equation}
where $A_i\nu$ denotes the pushforward of $\nu$ by $A_i$, and $\nu$ is non-atomic. The unique stationary measure \eqref{stationary} is called the Furstenberg measure. 

The Furstenberg measure is an important object to study in many fields. It is crucial for studying the asymptotic properties of random matrix products, such as the typical growth of vectors under random matrix products (Lyapunov exponents) \cite{BQ}. In quantum mechanics, the Furstenberg measure is important in the study of random Schr\"odinger operators \cite{BQ}, for example the smoothness properties of the Furstenberg measure is key to determining the regularity of the integrated density of states for the Anderson-Bernoulli model of random Schr\"odinger operators \cite{bourgain}. The Furstenberg measure is also a central ingredient in the dimension theory of self-affine sets and measures \cite{bhr,hr}.

The support of the Furstenberg measure $\nu$ is the limit set $\Lambda_\A$ \cite[Lemma 4.38]{BQ} therefore it is natural to study both the Furstenberg measure and the self-projective sets which they are supported on concurrently, for example the dimension theory of the Furstenberg measure is crucial for obtaining lower bounds on the Hausdorff dimension of self-projective sets. Indeed, recent progress on the dimension theory of the Furstenberg measure, beginning with \cite{hochman-sol} has enabled work on self-projective sets to begin in earnest, beginning with \cite{sol-tak}.

\section{Self-projective sets for subsets of $\SL$}

In the study of iterated function systems (IFS), the analysis of the attractor largely depends on three key properties of the IFS: 
\begin{itemize}
\item conformality/ non-conformality: whether the (maps in the) IFS are conformal or not,
\item contraction properties: whether the maps are uniformly contracting, have parabolic or critical points etc.,
\item separation/algebraic properties: how much overlap there is between copies of the attractor which make up the whole, whether there are any algebraic relations between distinct compositions of maps from the IFS.
\end{itemize}

\subsection{Conformality/ non-conformality.} When all of the maps in the IFS are conformal, the analysis of the attractor is considerably simpler, since it is made up of copies of the whole which, locally, have been distorted by the same amount in every direction. The simplest examples of attractors of conformal IFSs are self-similar sets, the attractors of an IFS composed of similarity contractions. In this section, we will be within the conformal regime, since the self-projective sets will be subsets of $\R\P^1 \cong S^1$ (any map on a one-dimensional space must necessarily be conformal). In fact, by associating matrices with their M\"obius action, self-similar sets in $\R$ can be seen as a special case of self-projective sets for subsets of $\SL$. In particular, the attractor of the self-similar iterated function system 
$$\{f_i(x)=r_ix+t_i\}_{\R \mapsto \R}$$
with $|r_i|<1$ and $t_i \in \R$ is a self-similar set which is (a bi-Lipschitz copy of) the self-projective set associated to the projective system
\begin{align}
\left\{ \begin{pmatrix} \sqrt{r_i}& \frac{t_i}{\sqrt{r_i}} \\ 0 & \frac{1}{\sqrt{r_i}}\end{pmatrix}\right\}, \label{commonrep}
\end{align}
which is a set of hyperbolic matrices with a common repelling fixed point. In \S \ref{higher} we will consider self-projective sets in higher dimensional settings which will introduce non-conformality into the picture.

\subsection{Contraction properties}

The best understood iterated function systems are those which are comprised of uniform contractions, sometimes called \emph{uniformly hyperbolic} iterated function systems. Mauldin and Urbanski \cite{mu} initiated the study of parabolic IFSs, IFSs which contain maps with a neutral fixed point, and studied them using a family of techniques known as \emph{inducing techniques}. This involves introducing an ``accelerated'' version of the IFS, which both captures the important properties of the attractor (such as its Hausdorff dimension) and has better properties than the original IFS (by being uniformly hyperbolic).

 While a priori self-projective sets cannot be associated to an IFS of contractions, we will see that there are important classes of self-projective sets which are analogous to the uniformly hyperbolic and parabolic IFSs. 

We begin by classifying matrices in $\SL$ with respect to their action on the projective line $\R\P^1$. Let $A\in\SL$ be a non-identity matrix. With slight abuse of notation, we will also use $A$ to denote the induced projective map $A\colon \R\P^1\to\R\P^1$. We say that $A$ is \emph{hyperbolic} if $A$ has two fixed points in $\R\P^1$, \emph{parabolic} if it has a unique fixed point in $\R\P^1$ and \emph{elliptic} otherwise.

Suppose $A$ is a hyperbolic matrix and let $\alpha_A,\beta_A\in\R\P^1$ denote its fixed points. We call $\alpha_A$ the \emph{attracting} fixed point of $A$ if it has the following property: if $x\in\R\P^1\setminus\{\beta_A\}$, then $A^n(x)\xrightarrow{n\to\infty}\alpha_A$; the other fixed point $\beta_A$ is called \emph{repelling}. This is equivalent to saying that $A$ maps a cone of $\R^2$ around its leading eigenvector compactly inside itself. 

So, semigroups generated by a single hyperbolic matrix define a uniformly contracting IFS on a part of the projective line, and in fact they are the only cyclic semigroups of $\SL$ with this property. In order to find the analogue of this behaviour for more general IFSs we require the following definition. Let $\A^n$ denote all possible products of $n$ matrices from $\A$.

\begin{defn}[Uniform hyperbolicity]
We say that $\A\subset\SL$ is uniformly hyperbolic if there exist $\lambda>1$ and $c>0$ such that for all $A \in \A^n$, $\norm{A} \geq c\lambda^n$.
\end{defn}

The simplest example of a uniformly hyperbolic set is one where all the matrices of $\A$ have strictly positive entries. Moreover, it is straightforward to see that if $\A$ is uniformly hyperbolic then the semigroup it generates contains only hyperbolic matrices. The converse of this last statement, however, is not true.

Uniformly hyperbolic sets were thoroughly studied by Avila, Bochi and Yoccoz \cite{aby} in order to decipher the properties of two-dimensional linear cocycles. There, they proved that $\A$ is uniformly hyperbolic if and only if there exists a finite union of open, connected subsets of $\R\P^1$ (called a \emph{multicone}) that is mapped compactly inside itself for all $A\in\A$ \cite[Theorem 2.2]{aby}. Thus, a uniformly hyperbolic set is an IFS which is uniformly contracting on a multicone. These types of IFSs will be the main focus of this section.

On the other end of the spectrum we have the following systems. 

\begin{defn}[Ellipticity]
We say that $\A$ is elliptic if the semigroup $G$ it generates contains an elliptic matrix or the identity.
\end{defn}

From our earlier comment we can see that if $\A$ is elliptic then it cannot be uniformly hyperbolic, and vice versa. Also, since elliptic matrices are rotations of the projective line, it is clear that an elliptic set of matrices induces an IFSs which exhibits no contracting properties. Later on we are going to present a more elaborate statement about the interaction between uniformly hyperbolic and elliptic sets.

The middling case between uniform hyperbolicity and ellipticity is the following, which was introduced by Jacques and Short \cite{JS}.

\begin{defn}[Semidiscrete]
We say that $\A$ is semidiscrete if the closure of the semigroup $\overline{G}$ it generates (in the usual topology) does not contain the identity.
\end{defn}

A reader familiar with Kleinian groups will notice that this definition is similar to the one typically used to define discrete \emph{subgroups} of $\SL$. However, the fact that we are working with semigroups instead drastically changes the implications of this definition. As a striking example we mention that semigroups generated by semidiscrete sets are not always discrete \cite[Section 3]{JS}.

Observe that if $\A$ is uniformly hyperbolic then it is semidiscrete, whereas an elliptic $\A$ is not. This justifies our description of semidiscreteness as the in-between case. We emphasise, however, that not every semidiscrete set is uniformly hyperbolic. There are various examples to see this (take a semidiscrete set which generates a non-discrete semigroup), but the simplest is to see that semidiscrete sets are allowed to contain parabolic matrices.

The action of semidiscrete sets on the projective line is difficult to describe. It can be shown that if $\A$ is semidiscrete, then there exists an open subset $D$ of $\R\P^1$ that is mapped inside itself by the projective action of $\A$. In most cases, $D$ will be mapped strictly inside itself (but not necessarily compactly) and so semidiscreteness can be thought of as ``non-uniform" contraction on a part of $\R\P^1$. But, there are examples of semidiscrete sets whose projective action fixes an open set and does not map any open set strictly inside itself. For a thorough analysis of the projective action of semidiscrete sets we refer to \cite[Section 4]{c}.

These three sets we defined do not exhaust all possibilities for subset of $\SL$, but cover most of them. In order to properly see how these definitions interact, we make the following shift in perspective: Fix $N\in\N$ and consider the parameter space $\SL^N$ with the product topology. The above definitions carry over to $\SL^N$ in the obvious way. That is, an $N$-tuple $\A\in\SL^N$ is uniformly hyperbolic if $\A$ is uniformly hyperbolic as a set of $N$ matrices, and similarly for elliptic and semidiscrete points of $\SL^N$.

This allows us to define the sets $\mathcal{H}$, of all uniformly hyperbolic $N$-tuples, $\mathcal{E}$ of all elliptic $N$-tuples and $\mathcal{S}$ of all semidiscrete $N$-tuples. The locus $\mathcal{H}$ is open but not connected, $\mathcal{E}$ is open and connected (\cite[Proposition A.3]{aby}) whereas $\mathcal{S}$ is neither closed, nor open and is not connected. Yoccoz (with a proof credited to Avila) \cite[Proposition 6]{yoccoz} showed that the complement of $\mathcal{H}$ is the closure of $\mathcal{E}$. This leads to the question of whether $\mathcal{H}$ and $\mathcal{E}$ share the same boundary in $\SL^N$. In \cite[Section 3]{aby} this was shown to be true when $N=2$, but \cite[Section 5]{c} showed that it is false for all $N\geq 3$. 

In order to understand the boundary of $\mathcal{H}$ one needs to study the structure of $\mathcal{S}$. This is because the locus $\mathcal{S}$ contains all of $\overline{\mathcal{H}}$, apart from certain ``simple" $N$-tuples \cite[Corollary 1.6]{c}. These exceptions are easy to study and all lie on the boundary of a particular connected component of $\mathcal{H}$. The significance of this result lies in the fact that $\mathcal{S}$ provides the general framework for the study of $N$-tuples which are very difficult to study through the lens of uniform hyperbolicity. These include $N$-tuples lie on the boundary of $\mathcal{H}$, but not on the boundary of any connected components of $\mathcal{H}$ (which certainly exist \cite[Example 6.5]{cthesis}).

This discussion raises several questions concerning the interaction between the hyperbolic and the semidiscrete loci of $\SL^N$. The most prominent is the following.

\begin{ques}\label{cocycle-q}
Can every semidiscrete $\A\in\SL^N$ be approximated by uniformly hyperbolic $N$-tuples?
\end{ques}

This question appeared in \cite[Question 2]{c} and is the latest version of a series of similar questions \cite[Question 4]{aby}, \cite[Section 16]{JS}, \cite[Question 4]{yoccoz}. In \cite[Section 6]{c} Question~\ref{cocycle-q} is split into two parts which suggests a possible route for tackling this problem. The difficulty of Question~\ref{cocycle-q} lies with the fact that it requires one to have good control on the limit set $\Lambda_\A$ when the perturbing the matrices in $\A$.

There are also several open questions about the topology of the connected components of the hyperbolic locus \cite[Section 6]{aby}. For example, it not known whether all connected components of $\mathcal{H}$ are unbounded, or whether the boundaries of two components can intersect.

\subsection{Algebraic properties} \label{section:alg}

 Consider an IFS $\mathcal{F}=\{f_i:X \to X\}_{i \in I}$ on some metric space $X$. Let $I^n$ denote the set of words of length $n$ with digits in $I$ and $I^*$ denote all words of finite length with digits in $I$. Given a finite word $i_1\ldots i_n \in I^n$ we let $f_{i_1\ldots i_n}=f_{i_1} \circ \cdots \circ f_{i_n}$. The best understood IFS are those for which the copies of the attractor which make up the whole are well separated. The most basic type of IFS which violates this is an IFS  $\mathcal{F}=\{f_i:X \to X\}_{i \in I}$ with \emph{exact overlaps}, i.e. with the property that $f_{\i}=f_{\j}$ for some $\i, \j \in I^*$. One of the most important open problems in the theory of iterated function systems is the \emph{exact overlaps conjecture} \cite{simon} for self-similar sets, which establishes the dimension of self-similar sets in the absence of exact overlaps. This conjecture is incredibly hard to solve, although recently it has been established in some special cases \cite{hochman1, varju, rapaportexact}. Typically one attempts to deal with the complications that arise from overlaps by imposing separation conditions on the IFS. For example, the IFS $\mathcal{F}$ satisfies the \emph{strong separation condition} if the points in the attractor are in one-to-one correspondence with points in the symbolic system $I^\N$, which is easier to study. The \emph{(strong) open set condition} is a slightly weaker assumption which allows some limited overlaps, meaning some points may have a non-unique symbolic coding. A substantially weaker separation condition which been popularised since the work of Hochman is the \emph{(strong) exponential separation condition} \cite{hochman1}, which allows for substantial overlaps between `pieces' $\{f_\i(X)\}_{\i \in I^n}$ of the attractor, provided that the pieces do not come too close to overlapping exactly `too fast'. Although a priori separation conditions do not make sense for self-projective sets (each projective map is a homeomorphism of projective space, hence a projective system is overlapping by definition) we will see that there are natural assumptions which can be made on the projective systems which are analogous to some of these separation conditions. 

If $\A$ freely generates the semigroup $G$, this is analogous to an absence of exact overlaps in the theory of iterated function systems. We will see an analogue of the exact overlaps conjecture can be formulated for self-projective sets in Conjecture \ref{conj}.

Since projective maps are homeomorphisms of $\R\P^1$, ``non-overlapping'' separation conditions are a little artifical in the study of self-projective sets. However, for projective systems such as uniformly hyperbolic systems, or more generally semidiscrete systems, which strictly contract a non-trivial subset of $\R\P^1$, analogous versions of these separation conditions can be formulated. Here we introduce a projective version of the open set condition.

\begin{defn}\label{osc}
We say that $\A$ satisfies the \emph{projective (strong) open set condition} if there exists an open set $U \subset \R\P^1$ (which intersects the limit set $\Lambda_\A$) such that $\bigcup_{A \in \A} A \cdot U \subset U$ where the union is disjoint.
\end{defn}

As an example, one can take any self-similar system which satisfies the open set condition (analogous to the above but where $U \subset \R^2$) and turn into a matrix system using \eqref{commonrep}. For example,
$$\left\{ \begin{pmatrix} \frac{\sqrt{2}}{2}&0\\0& \sqrt{2}\end{pmatrix}, \begin{pmatrix} \frac{\sqrt{2}}{2}&\frac{\sqrt{2}}{4}\\0& \sqrt{2}\end{pmatrix}\right\}$$
satisfies the projective open set condition for the open cone of positive directions that lie between the boundary vectors $v_1=(1,1)$ and $v_2=(0,1)$. Here we have transformed the self-similar IFS $\{f_1(x)=x/2, f_2(x)=x/2+1/2\}_{[0,1] \to [0,1]}$, which satisfies the open set condition for the open set $U=(0,1)$, by using \eqref{commonrep}. It is easy to see that  the Rauzy projective system \eqref{rauzy} satisfies the projective strong open set condition for the open set $U \subset \R\P^2$ where $U$ is the set of positive directions.

More generally, the following assumption \cite{hochman-sol} is analogous to the (strong) exponential separation condition, which first appeared in the work of Hochman \cite{hochman1} (note that it was not explicitly named the ``exponential separation condition'' there, but this terminology has been taken on since then, see e.g. \cite[Definition 1.9]{sol-tak}). We write $\A=\{A_1, \ldots, A_N\}$ and $\I=\{1, \ldots,N\}$.

\begin{defn}[Diophantine property]
We say that $\A$ is Diophantine if there exists $c>0$ such that if $\i, \j \in \I^n$ and $A_\i \neq A_\j$, then
$$\norm{A_\i-A_\j} \geq c^n.$$
We say that $\A$ is strongly Diophantine if $\A$ is Diophantine and generates a free semigroup.\footnote{This particular formulation of the Diophantine property was introduced in \cite{hochman-sol}, and the strong Diophantine property was introduced in \cite{sol-tak}. Note that in \cite{cj}, the Diophantine property is used to refer to being strongly Diophantine.}
\end{defn}

Unlike Definition \ref{osc}, this assumption makes sense for any projective system. It is known that when the matrices in $\A=\{A_1, \ldots, A_N\}$ have algebraic entries and generate a free semigroup, $\A$ is strongly Diophantine \cite[Proposition 4.3]{hochman-sol}.  Solomyak and Takahashi \cite[Theorem 1.2]{sol-tak} proved that if $\SL$ is treated as a subset of $\R^{4N}$ then Lebesgue almost all choices of sets of positive matrices $\A \subset \SL$ are strongly Diophantine. For a recent discussion of Diophantine properties in groups see \cite{sol-tak,dp2} and \cite{dp1} which highlights the connection to Diophantine numbers. 

\begin{ques} Until recently, it was not known whether the strong Diophantine property was equivalent to freeness. However, the recent examples \cite{baker1, baker2, bk} have shown that these properties are not equivalent by constructing self-similar iterated function systems (i.e. uniformly hyperbolic projective systems in $\SL$ where all the matrices have a common repelling fixed point, via the connection \eqref{commonrep}) which are free but not strongly Diophantine: they exhibit ``exponential condensation''. How easy is it to construct such examples in the more general self-projective setup?
\end{ques}

\begin{ques}\label{wsc}
Separation conditions such as the weak separation condition \cite{lau-ngai} and the finite type condition \cite{ngai-wang} have been introduced for self-similar sets, which permit special types of exact overlaps. Can such conditions be extended to semigroups of $\SL$?
\end{ques}

\subsection{Dimension of sets and measures}

Let $\hd X$ denote the Hausdorff dimension of a subset $X \subset \R\P^{d-1}$. Given a Borel probability measure $\mu$ on $\R\P^{d-1}$ and $x \in \mathrm{supp} \mu$ we define the local dimension of $\mu$ at $x$, wherever it exists, to be the limit
$$\locd \mu(x)= \lim_{r \to 0} \frac{\log \mu(B(x,r))}{\log r}.$$
If $\locd \mu(x)$ exists and is constant for $\mu$ almost every $x$, we say that $\mu$ is exact dimensional and we call the common value the exact dimension of $\mu$, denoted simply by $\dim \mu$. In this case many other notions of dimension of $\mu$ coincide and equal the exact dimension $\dim \mu$. 

We now describe the state of the art for the dimension theory of both self-projective sets and the Furstenberg measures which are supported on them. Since lower bounds on the Hausdorff dimension of self-projective sets require knowledge of the dimension of measures supported on the set, it is natural to begin with the dimension theory of Furstenberg measures, which has been developed first.

\subsubsection{Dimension of the Furstenberg measure}
Let $\A=\{A_1, \ldots, A_N\} \subset \SL$ generate a proximal and strongly irreducible semigroup $G$. Consider a probability measure $\mu=\sum_{i=1}^N p_i\log \delta_{A_i}$ which is fully supported on $\A$. Let $\nu$ be the Furstenberg measure \eqref{stationary} supported on $\R\P^1$. In \cite[Theorem 3.4]{hochman-sol} it is proved that $\nu$ is exact-dimensional, hence below we will just refer to the dimension of $\nu$ to mean its exact dimension.

In \cite{ledrappier}, Ledrappier computed the dimension of the Furstenberg measure $\nu$ in terms of the Furstenberg entropy of $\mu$. However since the Furstenberg entropy is generally difficult to compute, this result is less practically applicable. Therefore we instead turn to a more recent result of Hochman and Solomyak \cite{hochman-sol} which instead computes the dimension in terms of the random walk entropy. While this result is less general than that in \cite{ledrappier}, since it requires some assumptions on $\mu$, the random walk entropy is easier to compute, for example it coincides with the Shannon entropy of the measure in the case that the support of $\mu$ generates a free semigroup. 

Let $\chi(\mu)$ denote the almost sure value of the limit
$$\chi(\mu)= \lim_{n \to \infty} \frac{1}{n} \log \norm{X_n \cdots X_1},$$
where $(X_n)$ is a sequence of i.i.d random variables with distribution $\mu$. Let $H(\mu)=-\sum_{i=1}^N p_i \log p_i$ denote the Shannon entropy of $\mu$. Let $\mu^{\ast n}$ denote the $n$-th self convolution of $\mu$ in $G$, i.e. the distribution at time $n$ of the random walk on $G$ started at time zero at $1_G$ and driven by $\mu$. The random walk entropy of $\mu$ is then defined as
$$h_{RW}(\mu)=\lim_{n \to \infty} \frac{1}{n} H(\mu^{ \ast n}).$$
The limit is known to exist, see e.g. \cite{kv}, and quantifies the degree of non-freeness of the semigroup $G$ generated by $\mathrm{supp} \mu$: we always have $h_{RW}(\mu) \leq H(\mu)$
with equality if and only if the semigroup generated by $\mathrm{supp} \mu$ is generated freely by it.

\begin{thm}\label{HS}
Suppose that the semigroup $G$ generated by $\mathrm{supp} \mu$ is strongly irreducible, proximal, and Diophantine. Then
$$\dim \nu= \min\left\{1, \frac{h_{RW}(\mu)}{\chi(\mu)}\right\}.$$
In particular if $G$ is a free semigroup we have
\begin{align}
\dim \nu= \min\left\{1, \frac{H(\mu)}{\chi(\mu)}\right\}. \label{lyapdim}
\end{align}
\end{thm}

The right hand side of \eqref{lyapdim} is sometimes referred to as the Lyapunov dimension of $\mu$. This is always an upper bound on the exact dimension of $\mu$ and is conjectured to be equal to the exact dimension whenever $G$ is freely generated by $\mathrm{supp}(\mu)$.

The proof of Theorem \ref{HS} is an outgrowth of methods developed by Hochman \cite{hochman1, hochman2}, where an ``inverse theorem'', which provides conditions guaranteeing  entropy growth under convolution, is proved and applied to compute the Hausdorff dimension of self-similar measures on Euclidean space. In \cite{hochman-sol} these ideas are combined with linearisation techniques in order to prove Theorem \ref{HS}.

\subsubsection{Dimension of self-projective sets}

Let $\A \subset \SL$ be a finite set which generates the semigroup $G$. In this section $G$ will always be proximal, but may not be irreducible. We note that by \cite{FrMaSt} Theorem \ref{ls=attr} still holds when $d=2$ without assuming irreducibility, alternatively see \cite{cj}. We begin by introducing the critical exponent which, after taking the minimum with 1, will play the role of the Lyapunov dimension in the analysis of self-projective sets.

\begin{defn}[Zeta function and critical exponent]
Given a finite subset $\A \subset \SL$, we define the \emph{zeta function} $\zeta_\A: [0,\infty) \to \mathbb{R} \cup \{\infty\}$ as
\begin{equation}
\zeta_\A(s)= \sum_{n \in \N}\sum_{A \in \A^n} \norm{A}^{-2s}
\label{zeta}
\end{equation}
and the \emph{critical exponent} as
$$s_\A= \inf\{ s \geq 0: \zeta_\A(s)< \infty\}.$$
\end{defn}

The critical exponent is analogous to the critical exponent or \emph{Poincare exponent} in the study of Fuchsian groups. Under certain assumptions on $\A$, the critical exponent coincides with the (minimal) root of an appropriate \emph{pressure function}, which in some applications may be easier to use. In particular, in \cite[Theorem 6.3]{cj} it is shown that if $\A$ is irreducible and semidiscrete we have $s_\A=\delta_\A$  where 
$$\delta_\A= \inf\left\{\delta \geq 0: P_\A(\delta):=\lim_{n \to \infty} \left(\sum_{A \in \A^n} \norm{A}^{-2\delta}\right)^{\frac{1}{n}} \leq 1\right\}.$$

Solomyak and Takahashi \cite{sol-tak} proved the following.

\begin{thm}\label{ST}
Suppose $\A$ is uniformly hyperbolic, strongly Diophantine and $\Lambda_\A$ is not a singleton. Then 
$$\hd \Lambda_\A= \min\{1,s_\A\}.$$
\end{thm}

\begin{proof}
We note that the assumption that $\Lambda_\A$ is not a singleton, is to exclude the situation where all matrices in $\A$ have a common attracting fixed point $a$, so that $\Lambda_\A=\{a\}$ has dimension 0, while the critical exponent will be strictly positive. 

We sketch the proof of the upper bound, which is based on constructing a natural cover of $\Lambda_\A$ and provides a heuristic for why the dimension formula holds. Let $A \in \SL$ be hyperbolic. By the singular value decomposition there exist orthonormal bases $\{u_{A,1}, u_{A,2}\} \subset \R^2$ and $\{v_{A,1}, v_{A,2}\}\subset \R^2$ such that  $Au_{A,1}=\norm{A}v_{A,1}$ and $Au_{A,2}=\norm{A}^{-1}v_{A,2}$. In particular $A \cdot B(0,1)$ is an ellipse with semiaxes in the directions of $v_{A,1}$ and $v_{A,2}$ of lengths $\norm{A}$ and $\norm{A}^{-1}$ respectively. Therefore, now turning to the projective action, any subset of $\R\P^1$ which is $\epsilon$-bounded away from the direction of $u_{A,2}$ is contracted by $C_\epsilon\norm{A}^{-2}$ by the projective action of $A$, where $C_\epsilon$ is a constant which depends only on $\epsilon>0$. 

Now, since $\A$ is uniformly hyperbolic there exists a compactly contracting multicone $M$ and moreoever one can show that for all but finitely many $A \in G$ the multicone does not intersect the directions $u_{A,2}$. Hence we can find $\epsilon>0$ such that all but finitely many directions $u_{A,2}$ are $\epsilon$-bounded away from $M$.

Now, fix $\delta>0$. We can find a finite subset $H \subset G$ such that: (a) for all $A \in H$, $A \cdot M$ has diameter at most $\delta$ and (b) $\Lambda_\A$ is covered by $\bigcup_{A \in H} A \cdot M$. In particular the $s$-dimensional $\delta$-approximate Hausdorff measure $\mathcal{H}_\delta^s(\Lambda_\A)$ satisfies
$$\mathcal{H}_\delta^s(\Lambda_\A) \leq \sum_{A \in H} |A \cdot M|^s \leq  C_\epsilon\sum_n \sum_{A \in \A^n} \norm{A}^{-2s}<\infty$$
provided that $s>s_\A$. This proves that $\hd \Lambda_\A \leq s_\A$.

Note that this proof has only used the uniform hyperbolicity of $\A$, i.e. $\min\{1,s_\A\}$ is an upper bound on the Hausdorff dimension of $\Lambda_\A$ for any uniformly hyperbolic $\A$. 

On the other hand, the lower bound requires the strong Diophantine property. The lower bound follows by approximating the dimension of the self-projective set from below by the dimension of Furstenberg measures associated to subsystems of the form $\A^n$ and applying Theorem \ref{HS}.
\end{proof}

In \cite{cj} the authors proved the following generalisation of Theorem \ref{ST} to semidiscrete sets.

\begin{thm} \label{CJ}
Suppose $\A$ is semidiscrete, strongly Diophantine and $\Lambda_\A$ is not a singleton. Then 
$$\hd \Lambda_\A= \min\{1,s_\A\}.$$
\end{thm}

The proof of Theorem \ref{CJ} splits into two parts: one which handles the reducible case and the other which tackles the irreducible case. Up to conjugation, the reducible setting boils down to considering a few cases which are tackled by studying the corresponding Mobius semigroup. In the irreducible setting, one can show that the semigroup $G$ contains some hyperbolic element whose attracting fixed point is contained in an open interval which does not intersect $\Lambda_{\A^{-1}}$, where $\A^{-1}$ denotes the set of inverses of $\A$, and whose repelling fixed point is contained in an open interval which does not intersect $\Lambda_\A$. One can use this hyperbolic element to define an infinite uniformly hyperbolic subset $\A' \subset G$ with $\hd \Lambda_{\A'}=\hd \Lambda_{\A}$. In particular, the result follows by applying Theorem \ref{ST} to finite subsets of  $\A'$.

\subsubsection{Examples}

Finally we demonstrate how the dimension formula can be used to explicitly produce bounds or even exact values for the dimension of certain self-projective sets.

Our first example demonstrates that the slow contraction by parabolic elements immediately implies a lower bound on the dimension of self-projective sets generated by a set of matrices whose semigroup contains a parabolic element.

\begin{eg}
Suppose $\A$ is semidiscrete and $G_\A$ contains a parabolic element $P$. It follows that $\norm{P^n}=\mathrm{O}(n)$. Thus
\[
\zeta_\A(s)\geq \sum_{n=1}^\infty \norm{P^n}^{-2s}= C^{-2s}\sum_{n=1}^\infty n^{-2s},
\]
for some constant $C>0$, and the sum on the right hand side diverges for all $s\leq\tfrac{1}{2}$. So $s_\A\geq \tfrac{1}{2}$, which in turn implies that if $\A$ is strongly Diophantine and $\Lambda_\A$ is not a singleton $\hd \Lambda_\A\geq \tfrac{1}{2}$.
\end{eg}

Our next example shows that if a set $\A$ is semidiscrete but not discrete, then the dimension of the limit set $\Lambda_\A$ must be full.

\begin{eg}\label{eg-nondiscr}
Suppose $\A$ generates a semigroup which is semidiscrete but not discrete. There exists a sequence $(A_n)\subseteq G_\A$ that converges to some $A\in \SL$. So, there exists some $r>0$ such that $\norm{A_n}<r$, for all $n$. Hence, for all $s>0$
\[
\zeta_\A(s)\geq\sum_{n=1}^\infty \norm{A_n}^{-2s}\geq \sum_{n=1}^\infty\frac{1}{r^{2s}}.
\]
Thus $\zeta_\A(s)$ diverges for all $s>0$, implying that $s_\A=\infty$. In particular if $\A$ is strongly Diophantine and $\Lambda_\A$ is not a singleton, it follows that $\hd \Lambda_\A=1$.
\end{eg}

\subsubsection{Problems}

It is natural to ask to what extent is $\min\{1,s_\A\}$ equal to the Hausdorff dimension of $\Lambda_\A$ for $\A \subset \SL$. The ``exact overlaps conjecture'' \cite{simon} for self-similar sets can naturally be generalised to all self-projective sets. Since the notion of overlaps do not make sense in our setting (since our maps are homeomorphisms of $\R\P^1$), it makes more sense to refer to this more general version of the exact overlaps conjecture as the ``freeness conjecture'':

\begin{conj}[Freeness conjecture for self-projective sets] \label{conj}
Suppose $\A \subset \SL$ generates a free semigroup and $\Lambda_\A$ is non-empty and is not a singleton. Then
$$\hd \Lambda_\A=\min\{1,s_\A\}.$$
\end{conj}

Notice that Conjecture~\ref{conj} does not involve any assumptions on the contracting properties of $\A$. With some work it can be shown that it suffices to prove the conjecture under the extra assumption that $\A$ is semidiscrete. Let us sketch this reduction.

Assume that $\A$ is not semidiscrete, $\Lambda_\A\neq\emptyset$ and let $G$ be the semigroup $\A$ generates. We are going to prove that one of the following happens: either Conjecture~\ref{conj} is true; $G$ is not free; or $\Lambda_\A$ is a singleton.\\
If $\A$ is elliptic, then $G$ contains one of the following: an elliptic matrix of finite order; the identity; or an elliptic matrix of infinite order. In the first two cases $G$ is not free; and in the latter $\Lambda_\A=\R\P^1$ and $s_\A=+\infty$, so the conjecture holds trivially.\\
Suppose now that $\A$ is not elliptic. Since it is not semidiscrete either, by \cite[Theorem 14.1]{JS}, either $\A$ is reducible, or it maps a closed and connected, proper subset of $\R\P^1$ (that is not a singleton) inside itself. Let us first deal with the case where $\A$ is reducible.\\
Since $\A$ is not elliptic, all matrices in $\A$ are either hyperbolic or parabolic. Note that if $G$ contains two parabolic matrices $P_1,P_2$, with $P_1\notin\{{P_2}^n\colon n\in\N\}$, then $P_1$ and $P_2$ have a common fixed point in $\R\P^1$, which means that they commute and so $G$ is not free. Suppose that all parabolic matrices of $G$ (if they exist) are of the form $P^n$, for some parabolic $P\in\SL$ and some $n\in\N$. If all attracting fixed points of hyperbolic elements and the unique fixed point of parabolic elements (if they exist) coincide, then $\Lambda_\A$ is a singleton. Similarly, if all \emph{repelling} fixed points of hyperbolic elements and the unique fixed point of the parabolic element (if it exists) coincide, then $\A$ is semidiscrete, contrary to our assumption. The only case left to consider is when there exist hyperbolic matrices $A_1,A_2\in G$, with $\alpha_{A_1}=\beta_{A_2}$. If we also have that $\alpha_{A_2}=\beta_{A_1}$, then $A_1$ and $A_2$ commute and $G$ is not free. Otherwise, \cite[Lemma 3.11]{c} is applicable and yields that $\Lambda_{\{A_1,A_2\}}$ is closed and connected (and not a singleton), and so has Hausdorff dimension 1. Also, the set $\{A_1,A_2\}$ does not generate a discrete semigroup, meaning that $s_\A=\infty$, by the work carried out in Example~\ref{eg-nondiscr}. So, Conjecture~\ref{conj} holds trivially, concluding the case where $\A$ is reducible.\\
We now assume that the matrices in $\A$ map a closed and connected set $I\subsetneq \R\P^1$, that is not a singleton, inside itself. If all matrices in $\A$ map $I$ strictly inside itself, then $\A$ is semidiscrete by \cite[Theorem 7.1]{JS}, contrary to our assumption. So, there exists a matrix $A\in\A$ that fixes $I$ set-wise. Let $S$ be the subsemigroup of $G$ containing all the elements of $G$ that fix $I$. Since $\A$ is not elliptic, $\mathrm{Id}\notin S$. If $S=\{A^n\colon n\in\N\}$, then it is a semidiscrete semigroup, and \cite[Theorem 9.3]{JS} implies that $G\, \left(=G\setminus\{\mathrm{Id}\}\right)$ is also semidiscrete, which yields a contradiction. So, there exists (a necessarily hyperbolic) $B\in S$, with $B\notin\{A^n\colon n\in\N\}$, meaning that $A$ and $B$ commute and $G$ is not free.

Conjecture \ref{conj} tells us what Hausdorff dimension we should expect in the case that $\A$ freely generates a semigroup $G$, but the non-free case is also interesting.

\begin{ques}
Can one characterise the dimension of a self-projective set in the non-free case? For example, under the (generalised) weak separation or finite type conditions mentioned in Question \ref{wsc}? Can the dimension in the non-free case be characterised in terms of a zeta function which sums over the semigroup $G$ rather than \eqref{zeta} (which counts with multiplicity)? 
\end{ques}

\begin{ques}
What about other dimensions? e.g. the box dimension, Assouad dimension, Fourier dimension etc.
\end{ques}

\section{The higher dimensional setting}\label{higher}

Suppose that $\A\subset \mathrm{SL}(d,\R)$ and let $G$ be the semigroup generated by $\A$. We are going to prove Theorem~\ref{ls=attr} from the introduction. The techniques for our proof are inspired by \cite[Section 4.1]{BQ}.

For simplicity write $\Pi_\A=\{\pi\in\overline{\R G}\colon \mathrm{rank}(\pi)=1\}$. Recall that due to our standing assumptions $\Pi_\A$ is non-empty.

As a final piece of notation, for a set $D\subset\R^d$ we write $\P(D)$ for its projection into $\R\P^{d-1}$.

\begin{lma}\label{ls-compact}
If $\A$ is such that $\Pi_\A\neq\emptyset$, then $\Lambda_\A$ is a compact subset of $\R\P^{d-1}$.
\end{lma}

\begin{proof}
Observe that if $\pi\in\Pi_\A$, then $\norm{\pi}\neq0$ and $\pi'=\frac{1}{\norm{\pi}}\pi\in\Pi_\A$. Moreover, $\mathrm{Im}\ \pi' = \mathrm{Im}\ \pi$ and so $\Lambda_\A = \{\mathrm{Im} \ \pi \colon \pi\in\Pi_\A \ \text{with}\ \norm{\pi}=1\}$.\\
Define $B=\{\pi\in \overline{\R G} \colon \mathrm{rank}(\pi)\leq 1\}\cap \{\pi\in\overline{\R G}\colon \norm{\pi}=1\}$, which is a subset of $\overline{\R G}$. The set $\{\pi\in \overline{\R G} \colon \mathrm{rank}(\pi)\leq 1\}$ is closed because $\mathrm{rank}(\cdot)$ is a lower semi-continuous function on the set of $d\times d$ matrices. Also, $\{\pi\in\overline{\R G}\colon \norm{\pi}=1\}$ is compact, which implies that $B$ is compact. If we now define the map $F\colon  \overline{\R G}\to \R\P^{d-1}$ with $F(\pi)=\P(\mathrm{Im}\ \pi)$, the result follows from the fact that $F$ is continuous and $F(B)=\Lambda_\A$.
\end{proof}

\begin{proof}[Proof of Theorem~\ref{ls=attr}]
We first show that $\Lambda_\A$ contains the closure of all attracting fixed points of proximal elements of $G$. Note that since $\Lambda_\A$ is compact, by Lemma~\ref{ls-compact}, it suffices to prove this inclusion without the closure. Let $A$ be a proximal element of $G$ and let $\lambda$ be its leading eigenvalue. We can conjugate $\A$ by a matrix in $\mathrm{SL}(d,\R)$ so that $A$ is $\begin{pmatrix}
\lambda & \mathbf{0}\\
\mathbf{0} & B,
\end{pmatrix}$
where $B$ is in $\mathrm{GL}(d-1,\R)$ with eigenvalues of modulus strictly less than $\lvert \lambda \rvert$. Hence, $\pi=\lim\lambda^{-n}A^n$ is a rank one matrix and $\P(\mathrm{Im}\ \pi)$ is the attracting fixed point of $A$. \\
For the other inclusion, suppose $x\in\Lambda_\A$ and $x=\P(\mathrm{Im}\ \pi)$, for some $\pi\in\Pi_\A$. Since $\A$ is irreducible, we can find $A\in G$ so that $\pi A \pi\not\equiv0$. So, $A(\mathrm{Im}\ \pi)\not\subset \mathrm{Ker}\ \pi$. Write $\pi = \lim r_n A_n$ for $r_n\in\R$ and $A_n\in\A^*$, and note that $\pi A \pi=\lim r_n^2 A_nAA_n$. Let $U$ be a compact neighbourhood of $\P(\mathrm{Im}\ \pi)$ in $\R\P^{d-1}$. We can choose $U$ small enough so that $U\cap \P(\mathrm{Ker}\ \pi)=A(U)\cap \P(\mathrm{Ker}\ \pi) = \emptyset$. Then $\pi A\pi (U) = \P(\mathrm{Im}\ \pi)$, implying that $A_nAA_n$ maps $U$ into its interior for all $n$ large enough. Note that $U$ can be chosen so that there exists a closed convex cone $C\subset \R^d$, with $\P(C)=U$. So, $A_nAA_n$ maps $C$ into its interior, for all large $n$, and then the Krein-Rutman Theorem \cite{KrRu1950} implies that $A_nAA_n$ has a simple leading eigenvalue, with eigenvector in $C$. This means that $A_nAA_n$ is proximal with attracting fixed point in $U$. Since $U$ can be chosen arbitrarily small, we conclude that the attracting points of proximal matrices of $G$ are dense in $\Lambda_\A$.
\end{proof}

\begin{ques}\label{irred-q}
Does Theorem \ref{ls=attr} remain true if the assumption of irreducibility is dropped?
\end{ques}

\subsection{Contraction properties}

Let $\A \subset  \mathrm{SL}(d,\R)$ be finite. For $d \geq 3$, uniform hyperbolicity is replaced by the notion of 1-domination of $\A$. For a matrix $A \in \mathrm{SL}(d,\R)$ we let $\norm{A}=\alpha_1(A) \geq \alpha_2(A) \geq \cdots \geq \alpha_d(A)>0$ denote the singular values of $A$, that is, $\alpha_i(A)=\sqrt{\lambda_i(A^*A)}$ where $\lambda_1(A^*A) \geq \cdots \geq \lambda_d(A^*A)>0$ are the sequence of eigenvalues of $A^*A$, counted with multiplicity.

\begin{defn}[1-domination]
We say that $\A\subset \mathrm{SL}(d,\R)$ is 1-dominated if there exists $c>0$ and $\lambda>1$ such that for all $A \in \A^n$,
$$\frac{\alpha_1(A)}{\alpha_2(A)} \geq c\lambda^n.$$
\end{defn}

Note that in the case where $d=2$ we have $\alpha_2(A)=\norm{A}^{-1}$, hence this notion is equivalent to uniform hyperbolicity. There are analogous notions of $k$-domination, for $2 \leq k \leq d-1$. 

Bochi and Gourmelon obtained the following equivalent geometric criteria for 1-domination (in fact, for $k$-domination, although we only state the 1-dominated version below):

\begin{thm}\label{bogo}
$\A$ is 1-dominated if and only if there exists a non-empty subset $C \subset \R\P^{d-1}$ such that the closure of $C$ is mapped inside the interior of $C$ by $\A$ and $C$ does not contain the projection of a hyperplane.

Moreover, $C$ can be chosen to have a finite number of connected components with pairwise disjoint closures.
\end{thm}

\begin{ques} In \cite{aby} the authors underwent a thorough analysis of the structure of $\SL^N$ in terms of the topological properties of the hyperbolic and elliptic loci. What can be said about the structure of $\mathrm{SL}(d,\R)^N$ from this perspective? For example, what is the structure of the dominated loci (containing all 1-dominated tuples)? Is there a natural analogue of semidiscreteness which can be used to describe the boundary of the dominated loci?
\end{ques}

So Theorem~\ref{bogo} tells us that, just like their two-dimensional counterparts, 1-dominated sets induce a uniformly contracting IFS on a part of $\R\P^{d-1}$. These systems were studied by Barnsley and Vince in \cite{bv}, where they made the following definition.

\begin{defn}
A set $K_\A\subset\R\P^{d-1}$ is called an \emph{attractor} for $\A\subset\mathrm{SL}(d,\R)$, if $\A(K_\A)=K_\A$ and there exists an open set $U\subset\R\P^{d-1}$, containing $K_\A$, such that for every compact $B\subset U$, we have that $\A^k(B)$ converges to $K_\A$ in the Hausdorff metric, as $k\to\infty$. The set $U$ is called a \emph{basin of attraction} of $K_\A$.
\end{defn}

This definition of an attractor is similar to the usual definition in fractal geometry. By \cite[Theorem 2]{bv} if an attractor exists, then it is unique. Moreover, Barnsley and Vince show \cite[Theorem 1]{bv}, that a set $\A$ has an attractor $K_\A$ that avoids the projection of a hyperplane (i.e. there exists a hyperplane in $\R^d$ whose projection is transverse to $K_\A$) if and only if it is 1-dominated. 

The next result shows that when $\A$ is 1-dominated and irreducible, then the attractor and the limit set coincide.

\begin{thm}\label{bv-attr}
If $\A\subset\mathrm{SL}(d,\R)$ is 1-dominated and irreducible, then $K_\A=\Lambda_\A$.
\end{thm}

\begin{proof}
Since $\A$ is 1-dominated it contains proximal matrices. So Theorem~\ref{ls=attr} is applicable and thus it suffices to prove that the attracting fixed points of proximal elements are dense in the attractor. Let $U$ be a basin of attraction for $K_\A$ and consider $x\in K_\A$. Take $\epsilon>0$ small enough so that the closure of the open ball $D(x,\epsilon)\subset \R\P^{d-1}$ is contained in $U$. Then, by definition of the attractor and 1-domination of $\A$, there exists $A\in G$ such that $A\left(\overline{D(x,\epsilon)}\right)$ is contained in $D(x,\epsilon)$. So $A$ is a proximal matrix whose fixed point lies in $D(x,\epsilon)$. Since $\epsilon$ and $x$ were arbitrary, we can find attracting fixed points of proximal matrices arbitrarily close to any point of the attractor.
\end{proof}

Note that a positive answer to Question~\ref{irred-q} would imply that the irreducibility assumption can also be dropped from Theorem~\ref{bv-attr}.

If however a set $\A$ is not dominated then the limit set and the attractor are not necessarily the same. Consider the set $\A=\{A,B\}$, where 
\[
A=\begin{pmatrix}
2 & 0  &0\\
0 & 2 & 0\\
0 & 0& \frac{1}{4}
\end{pmatrix}
\quad \text{and}\quad 
B=\begin{pmatrix}
2 & 0 & 0\\
0 & 1 & 0\\
0 & 0 & \frac{1}{2}
\end{pmatrix}.
\]
Denote by $p\in\R\P^{2}$ the projection of the point $(1,0,0)$. Since $A$ and $B$ commute, we can easily see that any matrix in the semigroup generated by $\A$ is either $A^n$ or a proximal with the same eigenvectors as $B$ and attracting point $p$. This implies that if $\pi\in\Pi_\A$, then $\mathrm{Im}\ \pi = \mathrm{span}\{(1,0,0)\}$, which means that $\Lambda_\A=\{p\}$.\\
But $\Lambda_\A$ is \emph{not} an attractor for $\A$, since no compact neighbourhood of $p$ is contracted by $A$. Instead, the attractor $K_\A$ of $\A$ is the projection of the $x-y$ plane.\\
It is possible to modify this example so that $\A$ is irreducible, by choosing $A=\begin{pmatrix}
4E  &0\\
0& \frac{1}{4}
\end{pmatrix}$, for a suitable elliptic $E\in\SL$.

\subsection{Dimension of sets and measures} \label{dimhigher}

While the action of $\SL$ on $\R\P^1$ is conformal, the action of $\mathrm{SL}(d,\R)$ on $\R\P^{d-1}$ introduces non-conformality into the picture, and hence self-projective sets for $\A \subset \mathrm{SL}(d,\R)$ with $d \geq 3$ are more akin to self-affine sets (see e.g \cite{bhr,hr, feng}) than self-similar sets. This makes their dimension theory significantly more challenging, since the self-projective sets are composed of potentially heavily distorted copies of itself, which makes constructing covers for them quite subtle. Since the dimension theory of higher dimensional self-projective sets and the Furstenberg measures they support is still relatively limited, there are many interesting directions and open problems in this area.

A fundamental result in this direction is due to Rapaport \cite{rapaport}, who established the exact dimensionality of the Furstenberg measure in higher dimensions. 

\begin{thm} \label{rap}
Let $\A \subset \mathrm{SL}(d,\R)$ generate a strongly irreducible and proximal semigroup $G$. Let $\mu$ be fully supported on $\A$ and let $\nu$ be the Furstenberg measure \eqref{stationary}. Then $\nu$ is exact dimensional.
\end{thm}

Rapaport also showed that $\dim \nu$ satisfies a natural upper bound analogue to the right hand side of \eqref{lyapdim} in the $\mathrm{SL}(2,\R)$ case. Let us now specialise to the case $d=3$ and state the upper bound in that setting. We must first introduce the three Lyapunov exponents of the random matrix product. For $A \in \SL$ let $\alpha_1(A) \geq \alpha_2(A) \geq \alpha_3(A)$ denote the singular values of $A$. For $i \in \{1,2,3\}$ let $\chi_i(\mu)$ denote the almost sure value of the limit
$$\chi_i(\mu)=\lim_{n \to \infty} \frac{1}{n} \log \alpha_i(X_n \cdots X_1)$$
where $(X_n)$ is a sequence of i.i.d random variables with distribution $\mu$.

The Lyapunov dimension (cf. \eqref{lyapdim}) of the Furstenberg measure $\nu$  is given by
\begin{equation}
\ld \nu=
\begin{cases}
\frac{H(\mu)}{\chi_1(\mu)-\chi_2(\mu)} &\mbox{if } 0 \leq H(\mu) \leq \chi_1(\mu)-\chi_2(\mu) \\
1+\frac{H(\mu)-(\chi_1(\mu)-\chi_2(\mu))}{\chi_1(\mu)-\chi_3(\mu)} & \mbox{if } \chi_1(\mu)-\chi_2(\mu) \leq H(\mu) \leq 2\chi_1(\mu)-\chi_2(\mu)-\chi_3(\mu)\\
2 &\mbox{if } H(\mu) \geq 2\chi_1(\mu)-\chi_2(\mu)-\chi_3(\mu) \end{cases} 
\label{lyapdim*}
\end{equation}

Rapaport \cite{rapaport} showed that \eqref{lyapdim*} is an upper bound on the dimension $\dim \nu$ under the hypothesis of Theorem \ref{rap}. It is now a natural question whether a matching lower bound can be obtained analogous to Theorem \ref{HS}. This requires new ideas to deal with the nonconformality which was not present in the $\mathrm{SL}(2,\R)$ setting. By drawing on the entropy methods which were developed for self-affine measures \cite{bhr,hr}, Theorem \ref{HS} has recently been partially generalised to Furstenberg measures  \cite{rauzyme,li2023dimension} supported on $\R\P^2$. Let $ \SL_{>0}$ denote the subset of strictly positive matrices in $\SL$. The following result is from \cite{rauzyme}.

\begin{thm} \label{higherfursty}
Suppose $\mathrm{supp} \, \mu \subset  \SL_{>0}$ satisfies the projective strong open set condition and generates a Zariski dense semigroup $G$. Then $\dim \nu=\dim_{\mathrm{LY}} \nu$. 
\end{thm}

Under the weaker assumptions that $\mathrm{supp} \, \mu$ generates a Zariski dense and Diophantine semigroup, $\dim \nu$ was shown to satisfy a modified version of \eqref{lyapdim*} where all instances of the Shannon entropy are replaced by the (more difficult to compute) Furstenberg entropy \cite[Theorem 1.10]{li2023dimension}.

It should be possible to drop the assumption of positivity in Theorem \ref{higherfursty} and replace the projective strong open set condition by freeness, the latter of which is a challenging open problem. However the Zariski density cannot be replaced by e.g. proximality and strong irreducibility, which can be seen by taking generators from the  non-Zariski dense group $SO(2,1)$ which preserves the bilinear form $B(\mathbf{x},\mathbf{y})=x_1y_1+x_2y_2 -x_3y_3$, and thus preserves the curve $C=\{\mathbf{x} \in \R\P^2: x_1^2+x_2^2=x_3^2\}$, see\cite{rauzyme} for further details. 

Theorem \ref{higherfursty} can then be used to study the dimensions of self-projective sets in $\R\P^2$. Due to the nonconformality of the projective action, a copy $A(\Lambda_\A)$ (for $A \in \A$) of the self-projective set $\Lambda_\A$ may potentially be heavily distorted and so rather than covering it by one ball of radius comparable to $\mathrm{diam}(A(\Lambda_\A))$ it may be possible to cover it much more efficiently by a larger quantity of smaller balls. This leads to a natural generalisation of the critical exponent in this non-conformal setting, which is analogous to the affinity dimension in the theory of self-affine sets.

 For $s \geq 0$ we define a function $\phi^s: \SL \to \R_{>0}$ 
\begin{equation}
\phi^s(A)=
\begin{cases}
\left(\frac{\alpha_2(A)}{\alpha_1(A)}\right)^s &\mbox{if } 0 \leq s \leq 1 \\
\frac{\alpha_2(A)}{\alpha_1(A)}\left(\frac{\alpha_3(A)}{\alpha_1(A)}\right)^{s-1} & \mbox{if } 1 \leq s \leq 2\\
\left(\frac{\alpha_2(A) \alpha_3(A)}{\alpha_1(A)^2}\right)^s &\mbox{if } s \geq 2 \end{cases} 
\label{svf}
\end{equation}
This can be seen to be a direct analogue of the singular value function in the theory of self-affine sets once one recognises that $\frac{\alpha_2(A)}{\alpha_1(A)}$ describes the weaker contraction of $A$ (on most of $\R\P^2$) and $\frac{\alpha_3(A)}{\alpha_1(A)}$ describes the stronger contraction  (on most of $\R\P^2$). 

Given $\A \subset \SL$ we define the zeta function $\zeta_\A: [0,\infty) \to [0,\infty]$ by
\begin{equation} \label{zeta}
\zeta_\A(s)= \sum_{n=1}^\infty \sum_{A \in \A^n} \phi^s(A).
\end{equation}
Finally we define the (projective) \emph{affinity dimension} $s_\A$ to be the critical exponent of this series
$$s_\A=\inf\{s\geq 0: \zeta_\A(s)<\infty\}.$$

We are now ready to apply Theorem \ref{higherfursty} to establish results concerning the dimension of self-projective sets in $\R\P^2$. The following result was obtained in  \cite{rauzyme}.

\begin{thm} \label{setthm}
Suppose a finite set $\A \subset \SL_{>0}$ satisfies the projective strong open set condition and generates a Zariski dense semigroup $G$. Then $\hd \Lambda_\A=\min\{s_\A,2\}$.
\end{thm}

\begin{proof} [Sketch of proof.]Again we utilise the singular value decomposition, which in this setting allows us to provide for each $A \in G$ a pair of orthonormal bases $\{u_{A,i}\}_{i=1}^3$ and $\{v_{A,i}\}_{i=1}^3$ where $Au_{A,i}=\alpha_i(A)v_{A,i}$. Hence the image of a subset $B \subset \R\P^2$ of diameter $r$ which is at least $\epsilon$ separated from the plane $\mathrm{span}\{u_{A,1},u_{A,2}\}$ is contained in a ``projective ellipse'' with semiaxes lengths $O_\epsilon\left(\frac{\alpha_2(A)}{\alpha_1(A)}r\right)$ and $O_\epsilon\left(\frac{\alpha_3(A)}{\alpha_1(A)}r\right)$. This ``projective ellipse'' can be covered either by one projective ball of diameter $O_\epsilon\left(\frac{\alpha_2(A)}{\alpha_1(A)}\right)$ or roughly $\frac{\alpha_2(A)/\alpha_1(A)}{\alpha_3(A)/\alpha_1(A)}$ balls of diameter $O_\epsilon\left(\frac{\alpha_3(A)}{\alpha_1(A)}\right)$. Moreover, strict positivity of $\A$ implies that there exists $\epsilon>0$ such that the positive cone $K$ is at least $\epsilon$- separated from any plane $\mathrm{span}\{u_{A,1},u_{A,2}\}$. 

Fix $\delta>0$. We begin with the case that $s_\A \in [0,1)$. We will cover $\Lambda_\A$ using the first covering strategy described above. We can find a finite subset $H \subset G$ such that: (a) $\forall A \in H$, $A \cdot K$ has diameter at most $\delta$ and (b) $\Lambda_\A$ is covered by $\bigcup_{A \in H} A \cdot K$. The $s$-dimensional $\delta$-approximate Hausdorff measure $\mathcal{H}_\delta^s(\Lambda_\A)$ satisfies
$$\mathcal{H}_\delta^s(\Lambda_\A) \leq \sum_{A \in H} |A \cdot K|^s \leq \sum_n \sum_{A \in \A^n} O_\epsilon\left( \left(\frac{\alpha_2(A)}{\alpha_1(A)}\right)^s\right)<\infty$$
provided $s>s_\A$. 

Next suppose $s_\A \in [1,2)$. We will cover $\Lambda_\A$ using the second covering strategy described above. We can find a finite subset $H \subset G$ such that: (a) $\forall A \in H$, $A \cdot K$ is contained in a projective ellipse with semiaxes lengths $O_\epsilon(\frac{\alpha_2(A)}{\alpha_1(A)})$ and $O_\epsilon(\frac{\alpha_3(A)}{\alpha_1(A)}) \approx \delta$ and (b) $\Lambda_\A$ is covered by $\bigcup_{A \in H} A \cdot K$. We cover each such ellipse with balls of diameter $\delta$. The $s$-dimensional $\delta$-approximate Hausdorff measure $\mathcal{H}_\delta^s(\Lambda_\A)$ satisfies
$$\mathcal{H}_\delta^s(\Lambda_\A) \leq\sum_n \sum_{A \in \A^n} O_\epsilon\left( \frac{\alpha_2(A)/\alpha_1(A)}{\alpha_3(A)/\alpha_1(A)} \left(\frac{\alpha_2(A)}{\alpha_1(A)}\right)^s\right)<\infty$$
provided $s>s_\A$. 

For the lower bound, we combine the variational principle of Cao, Feng and Huang \cite{cfh} and recent results of Morris and Sert \cite{ms} to show that the Zariski density of $G$ implies that we can construct measures on `subsystems'  $\A^n$ whose associated Furstenberg measures have Lyapunov dimension that approximate the affinity dimension arbitrarily well. We then apply Theorem \ref{higherfursty} to deduce that the affinity dimension can be approximated arbitrarily well from below by the Hausdorff dimension of Furstenberg measures supported on $\Lambda_\A$, completing the proof.

\end{proof}

Note that Theorem \ref{setthm} cannot be applied to the Rauzy gasket, due to the fact that \eqref{rauzy} are not strictly positive matrices. However, by employing approximation techniques similar to those used to prove Theorem \ref{CJ}, the Hausdorff dimension of the Rauzy gasket can be recovered. The following result was obtained independently in \cite{rauzyme} and \cite{jiao2023dimension,li2023dimension}.

\begin{thm}
Let $\A$ be given by \eqref{rauzy} so that $R=\Lambda_\A$ is the Rauzy gasket. Then 
$$\hd R= \min\{2,s_\A\}=s_\A=\inf \left\{s>0:  \sum_{n=1}^\infty \sum_{A \in \A^n}
\left(\frac{\alpha_2(A) \alpha_3(A)}{\alpha_1(A)^2}\right)^s<\infty\right\}.$$
\end{thm}

\begin{proof} 
For the upper bound, although $\A$ itself isn't a subset of $\SL_{>0}$, we have that the set $\Gamma=\{A_i^nA_j\}_{i \neq j, n \in \N}$ can be simultaneously conjugated to a set of strictly positive matrices. Moreover $\Gamma$ has the property that for some compact subset $K'$ of $K$, each $A \in \Gamma$ maps $A(K) \subset K'$. Hence the upper bound on $\hd \Lambda_\Gamma$ follows very similarly to Theorem \ref{setthm} and so we get $\hd R=\hd \Lambda_\Gamma=\leq s_\Gamma$ since $R \setminus \Lambda_\Gamma$ is a countable sequence of points. On the other hand, for the lower bound one can show that any sufficiently large finite subset $\Gamma' \subset\Gamma$ generates a semigroup which is Zariski dense in $\SL$, hence the lower bound from Theorem \ref{setthm} can be applied to show that $\dim \Lambda_{\Gamma'}=s_{\Gamma'}$ . Moreover, it can be shown that $\sup_{\textnormal{finite  } \Gamma' \subset \Gamma} s_{\Gamma'}=s_{\Gamma}=s_\A$ which completes the proof.
\end{proof}

\subsubsection{Problems}

\begin{ques} Can Theorem \ref{rap} be generalised to the setting where one only assumes Zariski density, the Diophantine property and freeness of $G$? Or even more generally to cover the case where $G$ is Diophantine and Zariski dense but maybe not free (where the Lyapunov dimension is modified by replacing the Shannon entropy by the random walk entropy)? Equivalently, can the Furstenberg entropy be shown to coincide with the random walk entropy under the hypothesis of \cite[Theorem 1.10]{li2023dimension}?
\end{ques}

The dimension theory of self-projective sets in higher dimensions begs for a thorough investigation. The dimension theory of self-affine sets \cite{bhr,hr} offers up some natural questions in this setup, particularly in the 1-dominated case and the ``boundary'' of the 1-domination locus.

\begin{ques}Can it be shown that the ``generic'' or ``almost sure'' formula for the Hausdorff dimension of a self projective set $\Lambda_\A$ for $\A \subset \mathrm{SL}(d,\R)$ is given by the projective affinity dimension, analogous to the way in which the affinity dimension is the almost sure value for the Hausdorff dimension of a self-affine set \cite{falconer}?  If yes, under what assumptions on the set of matrices can this formula be made ``sure''? What can be said about to ``exceptional'' attractors for which such a formula does not apply; can they be treated using different techniques and are phenomena such as ``dimension gaps'' e.g. distinct Hausdorff and box dimensions typical for these attractors? 
\end{ques}

\begin{ques} What about the dimension theory of natural stationary measures and limit sets for the action of $\A$ on the Grassmannian $\mathrm{Gr}(d,k)$?
\end{ques}

\end{document}